\documentclass[a4paper,11pt]{amsart}
\usepackage{graphicx}
\usepackage{latexsym,amssymb,amsmath,amsfonts}
\usepackage[english]{babel}
\usepackage{enumerate}
\usepackage[T1]{fontenc}
\usepackage[latin1]{inputenc}
\usepackage[arrow,curve,color,line]{xypic}
\usepackage{xcolor}
\usepackage{hyperref}
\usepackage{stmaryrd}
\usepackage[active]{srcltx}
\newcommand{\op}{\llbracket}
\newcommand{\cl}{\rrbracket}
\def\pv#1{\ensuremath{\mathsf{#1}}}
\def\Om#1#2{\ensuremath{\overline{\Omega}_{#1}{\pv{#2}}}}
\def\oms#1#2{\ensuremath{{\Omega}^\sigma_{#1}{\pv{#2}}}}
\def\omo#1#2{\ensuremath{{\Omega}^\omega_{#1}{\pv{#2}}}}
\newcommand{\Sub}[1]{{\sf Sub}(#1)}

\newtheorem{Thm}{Theorem}
\newtheorem{Prop}[Thm]{Proposition}
\newtheorem{Lemma}[Thm]{Lemma}
\newtheorem{Cor}[Thm]{Corollary}
\begin{document}

\title[Reducibility vs. definability]{Reducibility versus definability for pseudovarieties of
  semigroups}

\author{J. Almeida}%
\address{CMUP, Dep.\ Matemática, Faculdade de Ciências, Universidade do
  Porto, Rua do Campo Alegre 687, 4169-007 Porto, Portugal}
\email{jalmeida@fc.up.pt}

\author{O. Kl\'\i ma}%
\address{Dept.\ of Mathematics and Statistics, Masaryk University,
  Kotlá\v rská 2, 611 37 Brno, Czech Republic}%
\email{klima@math.muni.cz}

\begin{abstract}
  It is easy to show that a pseudovariety which is reducible with
  respect to an implicit signature $\sigma$ for the equation $x=y$ can
  also be defined by $\sigma$-identities. We present
  several negative examples for the converse using signatures in which
  the pseudovarieties are usually defined. An ordered example issue
  from the extended Straubing-Thérien hierarchy of regular languages
  is also shown to provide a positive example for the inequality $x\le
  y$.
\end{abstract}


\makeatletter%
\@namedef{subjclassname@2010}{%
  \textup{2010} Mathematics Subject Classification}%
\makeatother


\maketitle


\section{Introduction}
\label{sec:intro}

Drawing motivation and problems from theoretical computer science,
specially from the theory of finite automata and regular languages,
the study of finite semigroups has led to substantial developments
since the 1960's. The connections between the two areas were
formalized in seminal work of Eilenberg
\cite{Eilenberg:1974,Eilenberg:1976} where, in particular, the
relevant classification of finite semigroups that emerged is in terms
of the so-called pseudovarieties. Indeed, through Eilenberg's
correspondence, pseudovarieties of semigroups are associated, via
syntactical recognition, to classes of languages (varieties) with
natural closure properties. Several combinatorial operations on
varieties of languages have been shown to correspond to algebraic
constructions on pseudovarieties of semigroups and the general aim is
to decide membership in a variety by deciding membership in the
corresponding pseudovariety. Running through this general program,
pseudovarieties of semigroups are often defined as smallest
pseudovarieties generated by a given class of semigroups, constructed
by applying some algebraic operator on semigroups from given
pseudovarieties. While this process does not in general preserve
decidability of the membership problem
\cite{Albert&Baldinger&Rhodes:1992, Auinger&Steinberg:2001b}, the
search for stronger hypotheses on the given pseudovarieties to
guarantee decidability for the resulting pseudovariety seems a
worthwhile endeavor. Many works in this direction have been developed,
starting with deep results of Ash~\cite{Ash:1991} on the pseudovariety
of all finite groups and various attempts of extending it
\cite{Almeida:1999b, Almeida&Steinberg:2000a}, more or less successful
depending on the algebraic operator under consideration. The case of
the semidirect product was initially based on a result
\cite{Almeida&Weil:1996} in whose proof a gap was found and which
remains to be filled (see the discussion
in~\cite[Chapter~3]{Rhodes&Steinberg:2009qt}).

In the approach considered in~\cite{Almeida&Steinberg:2000a}, several
properties of pseudovarieties of semigroups are considered depending
on an enriched algebraic signature for finite semigroups given by an
implicit signature: besides multiplication, a set of other operations
commuting with homomorphisms is taken into account. A basic property
is whether the signature is sufficiently rich to define the
pseudovariety (definability). Another important property, which is
related with the work of Ash~\cite{Ash:1991}, is whether the signature
is sufficient to witness solutions, modulo the given pseudovariety, of
systems of equations with constraints in finite semigroups
(reducibility). Even for the system of equations consisting of the
single equation $x=y$, where $x$ and $y$~are variables, it is easy to
show that reducibility implies definability
\cite{Almeida&Steinberg:2000a}. Prior to this work, nowhere in the
literature seems to be an example showing that the converse does not
hold. Many pseudovarieties have been shown to be reducible
\cite{Almeida&Costa&Zeitoun:2005b, Costa&Nogueira:2009,
  Almeida&Costa&Zeitoun:2015b} with respect to the signature in which
they are naturally defined, but proofs are very much dependent on the
properties of the specific pseudovarieties.

The aim of this paper is to understand the relationship between
definability and reducibility. We show, that for simple
pseudovarieties of semigroups usually defined within a certain
signature, reducibility also holds. In contrast, we present several
negative examples, of pseudovarieties which are definable in a natural
signature but not reducible with respect to it. The examples are drawn
from three natural classes of semigroups: commutative semigroups,
groups, and completely regular semigroups. The technique to establish
the negative results involves choosing a suitable regular language for
which the syntactic congruence is tight enough to have simple to
handle classes and allow a combinatorial analysis of the desired
witnesses.

\section{Preliminaries}
\label{sec:prelims}

The reader is referred to standard references \cite{Almeida:1994a,
  Almeida:2003cshort, Pin:1986;bk, Rhodes&Steinberg:2009qt} for
background on semigroups, pseudovarieties, and profinite semigroups.
Most of what we write about semigroups may equally well be established
for monoids but we usually stick with semigroups. We also consider the
extension of the theory of pseudovarieties of semigroups to ordered
semigroups \cite{Pin&Weil:1996b}.

Given a pseudovariety of (ordered or not) semigroups \pv V, the
pro-\pv V semigroup freely generated by a set $A$ is denoted \Om AV.
The pseudovarieties of all finite semigroups and of all finite monoids
are denoted, respectively, \pv S and~\pv M.

An \emph{implicit signature} is a set in which each element belongs to
some free profinite semigroup \Om AS, where $A$~is a finite set,
including binary multiplication. The elements $w$ of \Om AS may be
seen as $A$-ary (implicit) operations with a natural interpretation
$w_S:S^A\to S$ on each profinite semigroup $S$: for each function
$\varphi:A\to S$, we put $w_S(\varphi)=\hat\varphi(w)$, where
$\hat{\varphi}$ is the unique extension of~$\varphi$ to a continuous
homomorphism $\Om AS\to S$, the existence of which amounts to the
universal property defining free profinite semigroups. Even when
restricted to finite semigroups, this interpretation is injective and
produces exactly those operations that commute with homomorphisms. In
particular, for an implicit signature $\sigma$, finite semigroups and
\Om AS are naturally viewed as $\sigma$-algebras and this is always
the structure of $\sigma$-algebras that we will consider on them. The
$\sigma$-subalgebra of~$\Om AS$ generated by the set $A$ of free
generators, denoted \oms AS, is easily seen to be the free
$\sigma$-algebra in the Birkhoff variety of $\sigma$-algebras
generated by~\pv S. The elements of~\oms AS are called
\emph{$\sigma$-words}. Note that, when $\sigma$ is reduced to binary
multiplication, $\sigma$-words are simply words, meaning elements of
the free semigroup $A^+$. In general, elements of~\Om AS are called
\emph{pseudowords}. The set $A^+$ is topologically dense in the metric
space \Om AS, that is, every pseudoword is the limit of some sequence
of words.

Some relevant and frequently encountered examples of pseudowords may
be described as follows:
$x^{\omega+k}=\lim_{n\to\infty}x^{n!+k}$ ($k\in\mathbb{Z}$) and
$x^{p^\omega}=\lim_{n\to\infty}x^{p^{n!}}$. %
These are just a few examples of pseudowords in one variable, of which
there are uncountably many. More precisely, let $\hat{\mathbb{N}}$
denote the profinite completion of the semiring
$(\mathbb{N},{+},{\cdot})$ of non-negative integers. The
``logarithmic'' mapping $\Om{\{x\}}M\to\hat{\mathbb{N}}$ sending the
generator~$x$ to~1 extends uniquely to an isomorphism of~$\Om{\{x\}}M$
with the additive semigroup of the profinite
semiring~$\hat{\mathbb{N}}$. Composition in~$\Om{\{x\}}M$, given by
$(u\circ v)_S=u_S\circ v_S$ ``logarithmically'' translates to
multiplication in~$\hat{\mathbb{N}}$. Note that the pseudoword
$x^\omega$ corresponds to the only nonzero additive idempotent
of~$\hat{\mathbb{N}}$. The additive semigroup ideal it generates is a
subsemiring isomorphic with the profinite completion
$\hat{\mathbb{Z}}$ of the usual ring of integers
$(\mathbb{Z},{+},{\cdot})$, under an isomorphism sending $\omega+1$
to~$1$; it should therefore lead to no confusion to abuse notation and
denote the inverse isomorphism by $\gamma\mapsto\omega+\gamma$.

Abusing notation, we will also denote by the same symbols the implicit
signatures consisting of pseudowords determined by a set of exponents
in~$\hat{\mathbb{N}}$ and binary multiplication. For instance, we
write $\omega=\{\_\vphantom{|}^\omega,\, \_\cdot\_\}$. In case it is
stated that $\gamma\in\hat{\mathbb{Z}}$, the notation is interpreted
as $\gamma=\{\_\vphantom{|}^{\omega+\gamma},\, \_\cdot\_\}$.

By a \emph{pseudoidentity} we mean a formal equality $u=v$ of
pseudowords $u,v\in\Om AS$ for some finite set $A$. In case
$u,v\in\oms AS$, we call $u=v$ a \emph{$\sigma$-identity}. A finite
semigroup $S$ \emph{satisfies} the pseudoidentity $u=v$ and we write
$S\models u=v$ if $u_S=v_S$; this notion and notation are extended to
classes of finite semigroups and sets of pseudoidentities by requiring
that every semigroup in the class satisfy every pseudoidentity in the
set. Sometimes, it is convenient to use the abbreviation $u=1$ to
stand for the pair of pseudoidentities $ux=u=xu$, where $x$ does not
occur in~$u$. The class of all finite semigroups that satisfy all
pseudoidentities in a given set $\Sigma$ of pseudoidentities is
denoted $\op\Sigma\cl$ and is easily seen to be a pseudovariety. In
fact, every pseudovariety is of this form \cite{Reiterman:1982}.
Pseudoinequalities and satisfaction by finite ordered semigroups are
defined similarly. Pseudovarieties of ordered semigroups are also
defined by pseudoinequalities \cite{Molchanov:1994,Pin&Weil:1996b}. A
pseudovariety that may be defined by $\sigma$-identities is said to be
\emph{$\sigma$-equational} or simply \emph{equational} in case
$\sigma$~consists only of binary multiplication.

We recall some notions, simplified to the context that interests us
here. They are taken from \cite{Almeida&Steinberg:2000a}, a paper to
which the reader is also referred for further motivation. Given a
pseudovariety \pv V, a finite semigroup $S$, elements $s,t\in S$, and
an onto continuous homomorphism $\varphi:\Om AS\to S$, by a \emph{\pv
  V-solution} of the equation $x=y$ for the triple $(S,s,t)$ we mean a
pair $u,v\in\Om AS$ such that $\pv V\models u=v$, $\varphi(u)=s$, and
$\varphi(v)=t$. For an implicit signature $\sigma$, the pseudovariety
\pv V is said to be \emph{$\sigma$-reducible} (for the equation $x=y$)
if, whenever there is a \pv V-solution for a triple $(S,s,t)$, there
is some \pv V-solution consisting of $\sigma$-words. This property is
independent of the chosen onto continuous homomorphism $\varphi:\Om
AS\to S$ (cf.~\cite[Proposition~4.1]{Almeida&Steinberg:2000a}). In
case $\sigma$ is reduced to binary multiplication, we call \emph{word
  reducible} a $\sigma$-reducible pseudovariety. Similar notions may
be considered for pseudovarieties of ordered semigroups by replacing
the equation $x=y$ by the inequality $x\le y$ and the condition $\pv
V\models u=v$ by $\pv V\models u\le v$.

It is easy to see that, if a pseudovariety \pv V is
$\sigma$-reducible, then it is $\sigma$-equational (see
\cite[Proposition~4.2]{Almeida&Steinberg:2000a} for the unordered
case, the ordered case being handled similarly).

\section{Positive examples}
\label{sec:positive}

In this section, we exhibit some examples of pseudovarieties that are
usually defined by $\sigma$-identities that turn out also to be
$\sigma$-reducible.

The simplest example is that of locally finite pseudovarieties \pv V,
in which, for each positive integer~$n$, there is a bound on the size
of $n$-generated members of~\pv V, that is, \Om AV is finite for every
finite set $A$. Such pseudovarieties are clearly equational, being
defined for instance by all word identities describing the
multiplication of a word representative of each element by each
generator in each semigroup \Om AV with $A$~finite. The following
result may be considered a simple exercise and is presented here as a
warmup.

\begin{Prop}\label{p:loc-fin}
  Every locally finite pseudovariety is word reducible.
\end{Prop}

\begin{proof}
  Consider a finite semigroup $S$, $s,t\in S$, and a continuous
  homomorphism $\varphi:\Om AS\to S$. Assume that the pair $(u,v)$ is
  a \pv V-solution of the equation $x=y$ for the triple $(S,s,t)$. Let
  $(u_n)_n$ and $(v_n)_n$ be sequences of words converging to the
  pseudowords $u$ and $v$, respectively. Let $\psi:\Om AS\to\Om AV$ be
  the natural projection, mapping each generator to itself. Since the
  topologies considered in~$S$ and \Om AV are discrete, for all
  sufficiently large $n$, we have $\varphi(u_n)=\varphi(u)$,
  $\psi(u_n)=\psi(u)$, and similarly for $v_n$ and $v$. As $\pv
  V\models u=v$ if and only if $\psi(u)=\psi(v)$, it follows that the
  pair $(u_n,v_n)$ is a \pv V-solution of the equation $x=y$ for the
  triple $(S,s,t)$ whenever $n$~is large enough.
\end{proof}

The argument of the preceding proof may be similarly applied to handle
arbitrary systems of equations. Except for the fact that only a
special type of systems, determined by finite directed graphs, were
considered in~\cite{Almeida&Steinberg:2000a}, a much stronger result
is~\cite[Theorem~4.18]{Almeida&Steinberg:2000a}.


The pseudovariety $\pv J^+=\op 1\le x\cl$ corresponds to the one-half
level in the Straubing-Th\'erien hierarchy
\cite[Proposition~8.4]{Pin&Weil:1994c}.\footnote{The syntactic order
  that we consider for a language $L\subseteq A^+$ is $u\le_Lv$ if,
  for all $x,y\in A^*$, $xuy\in L$ implies $xvy\in L$. Some authors
  \cite{Pin:1997} consider the opposite order, which naturally leads to
  reversed pseudoinequalities.} Since the basis is equational, our
goal is to prove that the pseudovariety is word-reducible (for $x\le
y$). First we must recall that $u\le v$ is satisfied in $\pv J^+$ if
and only if every finite subword of $u$ is also a subword of $v$. More
precisely, if we define, for any pseudoword $w\in\Om{A}{M}$, the set
of all finite subwords $\Sub{w}$ in the following way
\[
\Sub{w}=\{a_1a_2\dots a_n \in A^* \mid a_1, a_2, \dots ,a_n \in A,\
\phantom{sdsfdgfsfgfafdsadf}
\]
\[\phantom{sadfgsdagdfga}
\exists w_0,w_1,\dots, w_n\in\Om{A}{M} : w=w_0a_1w_1a_2 \dots
a_nw_n\},
\]
then we have $\pv J^+ \models u\le v$ if and only if $\Sub{u}\subseteq
\Sub{v}$.

\begin{Lemma}\label{l:equivalent-subword}
  Given a finite monoid $M$, a homomorphism $\varphi : \Om{A}{M}
  \rightarrow M$ and a pseudoword $w\in \Om{A}{M}$, there exists a
  finite subword $v$ of $w$ such that $\varphi(v)=\varphi(w)$.
\end{Lemma}

\begin{proof}
  There is a sequence of words $(w_n)_n$ converging to $w$ in
  $\Om{A}{M}$ such that $\varphi(w_n)=\varphi(w)$, $n\ge 1$. We
  consider the Cayley graph of $M$ with respect to~$A$, in which
  vertices are elements from $M$ and, for every $m\in M$ and $a\in A$,
  we have an edge from $m$ to $m\cdot \varphi(a)$ labeled by the
  letter~$a$.
  Thus, every $w_n$ labels a path from $1$ to
  $\varphi(w_n)=\varphi(w)$ and one can extract a simple path from
  this path also starting in $1$ and ending in $\varphi(w)$. If the
  sequence of labels of the edges in this simple path is $(a_1,\dots
  ,a_k)$, then the considered word $w_n$ can be written as
  $w_n=u_0a_1u_1\dots a_k u_k$ for some words $u_0,\dots, u_k$. The
  assumptions concerning the extracted simple path also implies that
  $k < |M|$ and $\varphi(a_1\dots a_k)=\varphi(w)$. Since there are
  only finitely many simple paths in the Cayley graph of $M$, in
  infinitely many cases the extracted simple paths for words $w_n$ are
  the same. In this way we obtain a word $a_1\dots a_k$, a label of a
  simple path from $1$ to $\varphi(w)$, and a subsequence $(w_{n_i})$
  of the sequence $(w_n)$, such that $w_{n_i}=u_{i,0}a_1u_{i,1}\dots
  a_k u_{i,k}$ for some appropriate words $u_{i,j}$. Now, by
  compactness, there is a strictly increasing sequence $(i_\ell)_\ell$
  such that, for each $j=0,\dots,k$, $(u_{i_\ell,j})_\ell$ converges
  to some pseudoword $\bar u_j$. We thus obtain a final subsequence of
  $(w_n)_n$ converging to $w$ which shows that $w$ can be factorized
  as $w=\bar{u}_0a_1\bar{u}_1a_2 \dots a_k\bar{u}_k$. Hence,
  $v=a_1\dots a_k$ is a finite subword of $w$ satisfying the required
  equality $\varphi(v)=\varphi(w)$.
\end{proof}

\begin{Prop}
  The pseudovariety $\pv J^+$ is word-reducible (for $x\le y$).
\end{Prop}

\begin{proof}
  Let $u,v\in \Om{A}{M}$ be such that $\pv J^+ \models u\le v$, and
  let $\varphi : \Om{A}{M} \rightarrow M$ be a homomorphism to a
  finite monoid. By Lemma~\ref{l:equivalent-subword}, there is a
  finite word $u'=a_1a_2\dots a_n$, with $a_1,a_2,\dots ,a_n \in A$,
  such that $u'$ is a subword of $u$ satisfying
  $\varphi(u')=\varphi(u)$. Since $\pv J^+ \models u'\le u$ and
  consequently $\pv J^+ \models u'\le v$, there is a factorization
  $v=v_0a_1v_1a_2 \dots a_n v_k$ with $v_0,v_1,\dots
  ,v_n\in\Om{A}{M}$. Now, if we replace each $v_i$ by a finite word
  $v'_i$ such that $\varphi(v'_i)=\varphi(v_i)$, then we obtain the
  finite word $v'=v'_0a_1v'_1a_2\dots a_nv'_n$. The constructed pair
  of words $u'$ and $v'$ have the following properties:
  $\varphi(u')=\varphi(u)$, $\varphi(v')=\varphi(v)$, and $u'\in
  \Sub{v'}$, whence $\pv J^+ \models u'\le v'$.
\end{proof}


The remainder of the paper presents several examples of
pseudovarieties that are $\sigma$-equational but not
$\sigma$-reducible, sometimes even not $\sigma'$-reducible for a
larger signature $\sigma'$.

\section{Commutative semigroups}
\label{sec:com}

Our first negative example is that of the equational pseudovariety
$\pv{Com}=\op xy=yx\cl$ of all finite commutative semigroups. It is
shown in \cite{Almeida&Delgado:2001} that \pv{Com} is
$(\omega-1)$-reducible, in fact for all finite systems of
$(\omega-1)$-word equations. We show that it is not
$\omega$-reducible, whence also not word reducible (for the equation
$x=y$). Let \pv{Ab} stand for the pseudovariety of all finite Abelian
groups.

\begin{Thm}\label{t:Com}
  No pseudovariety in the interval $[\pv{Ab},\pv{Com}]$ is
  $\omega$-reducible.
\end{Thm}

\begin{proof}
  Let \pv V be a pseudovariety such that $\pv{Ab}\subseteq\pv
  V\subseteq\pv{Com}$. We exhibit a finite semigroup $S$, elements
  $s,t\in S$ and an onto continuous homomorphism
  $\varphi:\Om{\{x,y\}}S\to S$ such that there is a \pv{Com}-solution
  of the equation $x=y$ for the triple $(S,s,t)$ but no such solution
  exists in $\omega$-words. Let $X=\{x,y\}$ and $A=\{a,b\}$. We take
  $S$ to be the syntactic semigroup over the alphabet $A$ of the
  language $L=\bigl((a^2b)^2\bigr)^*\cup\bigl((ab^2)^2\bigr)^*$ and we
  denote by $[w]$ the syntactic class of a word $w\in\{a,b\}^+$. Let
  $s=[bab^2]$, and $t=[a^2ba]$. A standard calculation shows that
  $s=bab^2\bigl((ab^2)^2\bigr)^*$, $t=\bigl((a^2b)^2\bigr)^*a^2ba$,
  and
  \begin{equation}
    [ab^2]^{\omega-1}=[ab^2],\quad 
    [a^2b]^{\omega-1}=[a^2b].
    \label{eq:omm1}
  \end{equation}
  The continuous homomorphism $\varphi$~is
  defined by letting $\varphi(x)=[a]$ and $\varphi(y)=[b]$. Note that,
  for a word $w\in \{a,b\}^+$, the set $\varphi^{-1}([w])$ is the
  topological closure of $\eta^{-1}([w])$, where $\eta$~is the
  restriction of $\varphi$ to~$X^+$, that is, essentially the
  syntactic homomorphism of the language $L$ up to the change of
  letters $x\leftrightarrow a$, $y\leftrightarrow b$.

  We claim that there is no \pv V-solution of the equation $x=y$ for
  the triple $(S,s,t)$ in $\omega$-words. %
  Let $u,v\in\Om XS$ be pseudowords such that $\varphi(u)=s$ and
  $\varphi(v)=t$. The above description of the syntactic class $s$
  shows that %
  $u\in\overline{yxy^2\bigl((xy^2)^2\bigr)^*}
  =yxy^2\langle(xy^2)^2\rangle^1$, where $\langle w\rangle$ denotes
  the closed subsemigroup of~$\Om XS$ generated by~$w$; similarly, we
  have %
  $v\in\langle(x^2y)^2\rangle^1 x^2yx$.
  
  Let $\pi:\Om XS\to\Om X{Ab}$ be the natural continuous homomorphism,
  mapping each free generator to itself. It is well known that
  the profinite group $\Om X{Ab}$ is isomorphic with the product of
  two copies of the additive group $\hat{\mathbb{Z}}$ and we identify
  it with this product.

  Note that, for $w\in\Om XS$, since $\pi(w^\omega)$ is an idempotent,
  it is the identity element of the group
  $\hat{\mathbb{Z}}\times\hat{\mathbb{Z}}$. By induction on the
  construction of an $\omega$-word from the generators, it follows
  that $\pi(\omo XS)\subseteq\mathbb{N}\times\mathbb{N}$, the reverse
  inclusion being obvious.

  For each $w\in\Om XS$, we let $(|w|_x,|w|_y)=\pi(w)$. By the above
  discussion, there exist $\alpha,\beta\in\hat{\mathbb{Z}}$ such that
  $$|u|_x=2\alpha+1,\ %
  |u|_y=4\alpha+3,\ %
  |v|_x=4\beta+3,\ %
  |v|_y=2\beta+1.
  $$
  Assuming that $\alpha,\beta\in\mathbb{Z}$ and $\pv{Ab}\models u=v$,
  which entails $|u|_x=|v|_x$ and $|u|_y=|v|_y$, we obtain the system
  of equations $\beta=2\alpha+1$ and $\alpha=2\beta+1$, whose only
  integer solution is $\alpha=\beta=-1$. Hence, $u$ and $v$ cannot
  both be $\omega$-words, which establishes the claim. On the other
  hand, in view of the preceding calculations and~\eqref{eq:omm1}, the
  pair $\bigl(y(xy^2)^{\omega-1},(x^2y)^{\omega-1}x\bigr)$ is a \pv
  V-solution of the equation $x=y$ for the triple $(S,s,t)$. Hence,
  \pv V is not $\omega$-reducible.
\end{proof}

\section{Groups}

We say that a pseudovariety \pv V \emph{has infinite exponent} if it
satisfies no pseudoidentity of the form $x^{\omega+n}=x^\omega$, where
$n$~is a positive integer. Given a set $\sigma$ of pseudowords, we let
$\pv H_\sigma=\op w=1: w\in\sigma\cl$, which is a pseudovariety of
groups.

\begin{Thm}\label{t:groups}
  Let $\sigma$ be a set of binary implicit operations on the alphabet
  $X=\{x,y\}$ in which every element $w$ satisfies one of the
  following properties:
  \begin{enumerate}
  \item\label{item:groups-1} either $x^3$ or $y^3$ is a suffix of~$w$;
  \item\label{item:groups-2} both $xyx$ and $yxy$ are subwords of~$w$.
  \end{enumerate}
  Suppose further that the pseudovariety $\pv H_\sigma$ has infinite
  exponent. Then, no pseudovariety in the interval $[\pv
  H_\sigma\cap\pv{Ab},\pv H_\sigma]$ is
  $\sigma\cup\{\omega\}$-reducible.
\end{Thm}

\begin{proof}
  Let \pv H be a pseudovariety in the interval $[\pv
  H_\sigma\cap\pv{Ab},\pv H_\sigma]$. Note\ that $\pv G\models
  x^{\omega-1}y^\omega x^2=x$. We exhibit a semigroup $S$ and a pair
  of its elements $s,t$ such that $(x^{\omega-1}y^\omega x^2,x)$ is an
  $\pv H$-solution of the equation $x=y$ for the triple $(S,s,t)$,
  which has no $\pv H$-solution consisting of
  $\sigma\cup\{\omega\}$-words.
  We take $S$ to be the syntactic semigroup of the language
  $L=a^2a^+b^+a^2$ over the alphabet $A=\{a,b\}$. For each $w\in L$, a
  pair $(p,q)\in A^*\times A^*$ is a context of $w$, that is, $pwq\in
  L$, if and only if $p\in a^*$ and $q$ is the empty word. This
  means that all words from $L$ form one syntactic class $L=[w]$, for
  any $w\in L$, for instance for $w=a^3ba^2$. Since a pair $(a^3ba,1)$
  is a context of the word $a$ and it is not a context of any other
  word, we get $[a]=\{a\}$. Similarly, we can also see that
  $[a^4]=[a^3]$ and $[b^2]=[b]$.

  Now, consider the onto continuous homomorphism $\hat\varphi :
  \Om{X}{S} \rightarrow S$, which is the extension of the homomorphism
  $\varphi: X^* \rightarrow S$ uniquely given by $\varphi(x)=[a]$ and
  $\varphi(y)=[b]$. Further, we put
  $s=\hat\varphi(x^{\omega-1}y^\omega
  x^2)=[a]^{\omega-1}[b]^{\omega}[a]^2=[a^3ba^2]$ and
  $t=\hat\varphi(x)=[a]$, so that the pair $(x^{\omega-1}y^\omega
  x^2,x)$ is a \pv G-solution of the equation $x=y$ for the triple
  $(S,s,t)$, whence also an $\pv H$-solution. We show that there is no
  $\pv H$-solution of the equation $x=y$ for the triple $(S,s,t)$ in
  $\sigma\cup\{\omega\}$-words.

  Suppose that $u,v\in \Om{X}{S}$ are such that $\hat\varphi(u)=s$,
  $\hat\varphi(v)=t$ and $\pv H\models u=v$. Since $t=[a]=\{a\}$
  implies $\varphi^{-1}(t)=\{x\}$, we have also
  $\hat{\varphi}^{-1}(t)=\{x\}$ and, consequently, $v=x$. Now, we see
  that $\varphi^{-1}(s)=\{x^my^nx^2\mid m\ge 3,\, n\ge 1 \}$, because
  $s=[a^3ba^2]=L$. The pseudoword $u$ must, therefore, be a limit of
  words from the set $\{x^my^nx^2\mid m\ge 3,\, n\ge 1 \}$; in
  particular, $u$ does not contain $yxy$ as a subword. Hence, $u$~is
  equal (as an element of $\Om{X}{S}$) to a pseudoword of the form
  $x^\alpha y^\beta x^2$, where
  $\alpha,\beta\in\hat{\mathbb{N}}\setminus\{0\}$. Thus, \pv H
  satisfies the pseudoidentity $x^\alpha y^\beta x^2=x$.


  Now, suppose that $u$ is a $\sigma\cup\{\omega\}$-word. We claim
  that this assumption leads to a contradiction, namely that $\pv
  H_\sigma$ satisfies some identity of the form $x^{k+2}=x$, where
  $k\in\mathbb{N}$, which is contrary to the hypothesis that $\pv
  H_\sigma$ has infinite exponent and thereby concludes the proof. To
  prove the claim, consider an expression of~$u$ as a
  $\sigma\cup\{\omega\}$-word. Since $u$~is not a word, such an
  expression must be of the form $u=u_0\psi(w)u_1$, where $u_0$ is
  another $\sigma\cup\{\omega\}$-word, $u_1$ is a word,
  $w\in\sigma\cup\{x^\omega\}$, and $\psi$ is a continuous
  endomorphism of~\Om XS. If $|u_1|\ge2$, then $x^2$ is a suffix
  of~$u_1$ and the claim holds since $\pv H\models u_0\psi(w)=u_2$ for
  some word $u_2$ which, upon identification of the variables $x$ and
  $y$, reduces the pseudoidentity $u=v$ to an identity of the form
  $x^{k+2}=x$. Hence, we may assume that $|u_1|\le1$, so that
  $\psi(w)$ must end with the letter~$x$.

  Let $z$ be the last letter of~$w$, which entails that $x$~is the
  last letter of~$\psi(z)$. Note that $w$ cannot be $x^\omega$ for,
  otherwise, $x^3$ would be a suffix of~$u=x^\alpha y^\beta x^2$,
  which is clearly not the case. Suppose first that $w$~satisfies the
  condition~\eqref{item:groups-1} of the hypothesis, so that $z^3$~is
  a suffix of~$w$. If both letters $x$ and $y$ appear in~$\psi(z)$
  then, since $\psi(z)$~ends with $x$, $yx$~is a subword of~$\psi(z)$, so
  $yxy$~is a subword of~$\psi(z^2)$, whence also of $\psi(w)$ and
  of~$u$, which we know to be false. Hence, $\psi(z)$ must be of the
  form $x^\gamma$, with $\gamma\in\hat{\mathbb{N}}\setminus\{0\}$.
  Again, since $z^3$~is a suffix of~$w$, it follows that $x^3$ is a
  suffix of~$u$, which is not the case. It remains to consider the
  case where $w$~satisfies the condition~\eqref{item:groups-2} of the
  hypothesis. If $\psi(w)$~is a power of~$x$, then again $x^3$~is a
  suffix $\psi(w)$, whence of~$u$, which is false. Hence, both letters
  $x$ and $y$ intervene in~$\psi(w)$. Since both $xyx$ and $yxy$~are
  subwords of~$w$, it follows $yxy$ is a subword of~$u$, which is
  false. This concludes the proof of the claim.
\end{proof}

Of course, the dual of the theorem, where ``suffix'' is replaced by
``prefix'' in condition~\eqref{item:groups-1} is also valid. In case
$\sigma$~is a singleton set, we may combine these two results to
obtain a stronger result.

\begin{Cor}
  \label{c:groups}
  Let $u\in\Om{\{x,y\}}S$ be a pseudoword such that the group
  pseudovariety $\pv H_u=\op u=1\cl$ has infinite exponent. Then no
  pseudovariety in the interval $[\pv H_u\cap\pv{Ab},\pv H_u]$ is
  $\{u,\omega\}$-reducible.
\end{Cor}

\begin{proof}
  If any of the conditions \eqref{item:groups-1}, or its dual, or
  \eqref{item:groups-2} of Theorem~\ref{t:groups} is satisfied by~$u$,
  then we may apply the theorem to obtain the desired non-reducibility
  property. Otherwise, up to exchanging variables, we may assume that
  $u$ is a pseudoword of the form $x^my^\alpha x^n$, where
  $m,n\in\{1,2\}$. But then, substituting $y$ by~$x^\omega$ in the
  pseudoidentity $u=1$, we see that $\pv H_u\models x^{m+n}=1$, which
  contradicts the assumption that $\pv H_u$~has infinite exponent.
\end{proof}

In particular, the pseudovarieties \pv G and $\pv G_p=\op
x^{p^\omega}=1\cl$, where $p$~is prime, are not $\omega$-reducible:
take $\sigma=\{x^\omega\}$ for the first of these pseudovarieties and
$\sigma=\{x^{p^\omega}\}$ for the latter, which is in fact not
$\{\omega,p^\omega\}$-reducible. Many other examples can be
considered. By
\cite[Theorem~3.2]{Almeida&Margolis&Steinberg&Volkov:2010}, every
extension-closed pseudovariety of groups is of the form $\pv H_u$ for
some $u\in\Om{\{x,y\}}S$, and therefore Corollary~\ref{c:groups}
yields that, if nontrivial, then the pseudovariety $\pv H_u$ is not
$\{u,\omega\}$-reducible. Note that, for the pseudovariety $\pv
G_{\mathrm{sol}}$ of all finite solvable groups, concrete pseudowords
$u\in\Om{\{x,y\}}S$ such that $\pv H_u=\pv G_{\mathrm{sol}}$ have been
hard to construct, with arguments that depend on the classification of
finite simple groups %
\cite{ Bandman&Greuel&Grunewald&Kunyavskii&Pfister&Plotkin:2003,
  Bray&Wilson&Wilson:2005}.
Another important pseudovariety of the form $\pv H_u$ is $\pv
G_{\mathrm{nil}}$, of all finite nilpotent groups, where $u=[x,_\omega
y]=\lim_{n\to\infty}[x,_ny]$, with the iterated commutator defined
recursively by $[s,t]=s^{\omega-1}t^{\omega-1}st$ and
$[s,_{n+1}t]=[[s,_nt],t]$. In this case, it is easy to see that one
may apply Theorem~\ref{t:groups} directly.

\section{Completely regular semigroups}
\label{sec:CR}

The aim of this section is to prove that the pseudovariety
$\pv{CR}=\op x^{\omega+1}=x\cl$, consisting of all completely regular
semigroups, is not $\omega$-reducible. Since our proof technique is
similar to the case of groups, we consider the more general case of
the pseudovarieties $\pv{CR}(\pv{H}_\sigma)=\pv{CR}\cap \bar{\pv
  H}_\sigma$ of all completely regular semigroups whose subgroups
belong to $\pv{H}_\sigma$, where $\pv{H}_\sigma$ has infinite exponent
and $\sigma$ satisfies some suitable combinatorial hypothesis to be
specified below. Note that $\pv{CR}$ can be obtained as
$\pv{CR}(\pv{H}_\omega)$, because $\pv{H}_\omega=\pv G$.

We say that a pseudoword $w$ has two \emph{disjoint occurrences} of
another pseudoword $u$ if there is a factorization of~$w$ in which $u$
appears twice as a factor. A word $u$ is said to appear as a factor
of~$w$ \emph{within bounded distance from the end} if some finite
suffix of~$w$ admits $u$ as a factor.

\begin{Thm}\label{t:cr}
  Let $\sigma$ be a set of binary implicit operations on the alphabet
  $X=\{x,y\}$ in which every element satisfies one of the
  following properties:
  \begin{enumerate}
  \item\label{item:cr-1} either $x^3$ or $y^3$ is a suffix of~$w$;
  \item\label{item:cr-2} each factor of length 4 of $w$ within bounded
    distance from the end has two disjoint occurrences in~$w$.
  \end{enumerate}
  Suppose further that the pseudovariety $\pv H_\sigma$ has infinite
  exponent. Then, no pseudovariety in the
  interval $[\pv{H}_\sigma\cap\pv{Ab},\pv{CR}(\pv{H}_\sigma)]$ is
  $\sigma\cup\{\omega\}$-reducible.
\end{Thm}

\begin{proof} 
  Let \pv V be an arbitrary pseudovariety from the interval $[\pv
  H_\sigma\cap\pv{Ab},\pv{CR}(\pv H_\sigma)]$. We claim that the
  pseudoidentity $(x^2y)^{\omega-1}(xy^2)^\omega(x^2y)^2=x^2y$ is
  valid in~$\pv{CR}$. To prove the claim
  without invoking the general solution of the $(\omega-1)$-word
  problem for~\pv{CR} \cite{Kadourek&Polak:1986,
    Almeida&Trotter:2001},
  let $T$ be a finite
  completely regular semigroup and $\varphi: \Om XS \rightarrow T$ be
  a continuous homomorphism. We let $p=\varphi(x)$ and
  $q=\varphi(y)$. Note that the elements $qp$ and $q^2p$ are $\mathcal
  L$-equivalent. Therefore, $(q^2p)^\omega$, an idempotent $\mathcal
  L$-equivalent to $qp$, is right neutral to $qp$. Thus we obtain $qp
  (q^2p)^\omega=qp$, where the left hand side can be written as $q
  (pq^2)^\omega p$. It follows that
  $(p^2q)^{\omega-1}(pq^2)^\omega(p^2q)^2=(p^2q)^{\omega-1} (p^2q)^2$,
  where the right hand side is equal to $p^2q$ in the completely
  regular semigroup $T$. Finally, since the pseudoidentity is valid in
  $\pv{CR}$, it is valid also in~\pv V.
  
  We proceed in a similar way as in the proof of
  Theorem~\ref{t:groups}. Let $S$ be the syntactic semigroup of the
  language $L=(a^2b)^2(a^2b)^+(ab^2)^+(a^2b)^2$ over the alphabet
  $A=\{a,b\}$. Note that each word $w$ from the language $L$ has a
  unique occurrence of the factor $b^2a^2$. Hence, if two words from
  $L$ overlap, then one is a suffix of the other. Therefore, for each
  $w\in L$, a pair $(p,q)\in A^*\times A^*$ is a context of $w$, if
  and only if $p\in (a^2b)^*$ and $q$ is the empty word. This means
  that $L$ forms one syntactic class. Since a pair
  $((a^2b)^3(ab^2)(a^2b),1)$ is a context of the word $a^2b$ and it is
  not a context of any other word, we get $[a^2b]=\{a^2b\}$. One can
  also check that $[a^2b]^4=[a^2b]^3$ and $[ab^2]^2=[ab^2]$.
  
  For $X=\{x,y\}$, we consider an onto continuous homomorphism
  $\hat\varphi : \Om{X}{S} \rightarrow S$, which is the extension of
  $\varphi: X^* \rightarrow S$ uniquely given by $\varphi(x)=[a]$ and
  $\varphi(y)=[b]$. We put
  \[s=\hat\varphi((x^2y)^{\omega-1}(xy^2)^\omega(x^2y)^2)
  =[a^2b]^{\omega-1}[ab^2]^{\omega}[a^2b]^2=[(a^2b)^3(ab^2)(a^2b)^2]\]
  and $t=\hat\varphi(x^2y)=[a^2b]$. Assume that there are
  $\sigma\cup\{\omega\}$-words $u,v\in
  \Omega^{\sigma\cup\{\omega\}}_X\pv S$ such that $\hat\varphi(u)=s$,
  $\hat\varphi(v)=t$ and $\pv V\models u=v$.

  We have $t=[a^2b]=\{a^2b\}$, which implies
  $(\hat\varphi)^{-1}(t)=\{x^2y\}$. So, we have $v=x^2y$. Furthermore,
  we see that
  $$\varphi^{-1}(s)=\{(x^2y)^m(xy^2)^n(x^2y)^2\mid m\ge 3, n\ge 1 \}\
  ,$$
  because $s=[(a^2b)^3(ab^2)(a^2b)^2]=L$. The pseudoword $u$ must be a
  limit of words from the set $\varphi^{-1}(s)$. Therefore, $u$ is of
  the form $(x^2y)^\alpha (xy^2)^\beta (x^2y)^2$, where
  $\alpha,\beta\in\hat{\mathbb N}$.

  Consider an expression of~$u$ as a $\sigma\cup\{\omega\}$-word.
  Since $u$~is not a word, such an expression must be of the form
  $u=u_0\psi(w)u_1$, where $u_0$ is another
  $\sigma\cup\{\omega\}$-word, $u_1\in X^*$ is a word,
  $w\in\sigma\cup\{x^\omega\}$, and $\psi$ is a continuous
  endomorphism of~\Om XS. We claim that $|u_1|\ge 4$, in which case we
  are able to proceed as in the proof of Theorem~\ref{t:groups}.

  Assume for a moment, that $|u_1|< 4$. In other words, we have
  $u_1\in\{1, y, xy, x^2y\}$. First, we discuss the case when $w$
  satisfies condition~(\ref{item:cr-1}). Let $z\in X$ be the last
  letter of~$w$, so that $z^3$ is a suffix of $w$. We can see that
  $|\psi(z)|\ge 2$, because $u$ does not contain a cube of a letter as
  a factor. In case $|\psi(z)|=2$, one can easily check that none of
  the possible alternatives $\psi(z)=x^2$, $\psi(z)=xy$, $\psi(z)=yx$
  or $\psi(z)=y^2$ can hold as $\psi(z)\psi(z)u_1$ is a suffix of
  $yx^2yx^2y$. Thus, we have $|\psi(z)|\ge 3$ and, therefore,
  $|\psi(z^3)u_1|\ge 9$. This means that $(\psi(z))^3$ contains as a
  factor the word $y^2x^2$. Since this factor has length~4 and the
  length of $\psi(z)$ is at least 3, we deduce that $y^2x^2$ is even a
  factor of $\psi(z)^2$. However, in such a case the factor $y^2x^2$
  has at least two disjoint occurrences in $\psi(z)^3$, which is not
  possible, as $u$ contains just one occurrence of the factor
  $y^2x^2$.
  
  Now, assume that $w$ satisfies condition~(\ref{item:cr-2}) and
  recall that $w$ is not a word. As in the first case, under the
  assumption that $|u_1|<4$, one can show that the word $y^2x^2$ is a
  factor of $\psi(w)$. This means that it is a factor of some
  $\psi(w')$ where $w'$ is a factor of $w$ of length~$4$ within
  bounded distance from the end. Since $w$ satisfies
  condition~(\ref{item:cr-2}), we deduce that $\psi(w)$ contains
  another disjoint occurrence of $y^2x^2$, which is a contradiction.
  This completes the proof of the claim that $|u_1|\ge 4$.
  
  We can reformulate the previous claim as follows. The pseudoword $u$
  is a product of a certain $\sigma\cup\{\omega\}$-word $u'$ of the
  form $(x^2y)^\alpha (xy^2)^\beta x^2$ and the finite word $yx^2y$.
  Next, we consider the continuous homomorphism $\varphi :\Om XS
  \rightarrow \Om XS$ given by $\varphi(x)=\varphi(y)=x$. Since $\pv V
  \models u=v$, we obtain $\pv V \models \varphi(u)=\varphi(v)$. Since
  $\pv{H}_\sigma\cap\pv{Ab} \subseteq \pv V \subseteq
  \pv{CR}(\pv{H}_\sigma)$, the pseudoidentity $\varphi(u)=\varphi(v)$
  must be valid in $\pv{H}_\sigma\cap\pv{Ab}$. The prefix of the left
  hand side, corresponding to the prefix $\varphi(u')$ of
  $\varphi(u)$, is equivalent over $\pv{H}_\sigma\cap\pv{Ab}$ to $x^k$
  for some non-negative integer $k$. This means that, over
  $\pv{H}_\sigma\cap\pv{Ab}$, the pseudoidentity
  $\varphi(u)=\varphi(v)$ is equivalent to a certain identity
  $x^kx^4=x^3$ for some non-negative integer $k$. However such
  pseudoidentity is not valid in $\pv{H}_\sigma\cap \pv{Ab}$, because
  $\pv{H}_\sigma\cap \pv{Ab}$ and $\pv{H}_\sigma$ satisfy the same
  unary pseudoidentities and the pseudovariety $\pv{H}_\sigma$ has
  infinite exponent, a contradiction.
  
  This means that the equation $x=y$ does not have a $\pv V$-solution
  for the triple $(S,s,t)$ consisting of $\sigma\cup\{\omega\}$-words.
  On the other hand, we saw that the equation $x=y$ has a $\pv
  V$-solution consisting of $(\omega-1)$-words for the same triple
  $(S,s,t)$. We conclude that $\pv V$ is not
  $\sigma\cup\{\omega\}$-reducible.
\end{proof}

Note that Theorem~\ref{t:cr} has some overlap with
Theorem~\ref{t:groups} but does not quite supersede it. We have not
succeeded in finding an analog of Corollary~\ref{c:groups} for
Theorem~\ref{t:cr}.

Examples of application of Theorem~\ref{t:cr} include \pv{CR},
$\pv{CR}(\pv G_p)$, and $\pv{CR}(\pv G_{\mathrm{nil}})$, where
$\sigma=\{u\}$, respectively with $u=x^\omega$, $u=x^{p^\omega}$, and
$u=[x,_\omega y]$.

Another example is obtained by taking
$u=\mu^\omega(x)=\lim_{n\to\infty}\mu^{n!}(x)$, where $\mu$~is the
Prouhet-Thue-Morse substitution, defined as the endomorphism of
$\{x,y\}^+$ such that $\mu(x)=xy$ and $\mu(y)=yx$. The length of the
word $\mu^n(x)$ is~$2^n$. Hence, by identification of the variables
$x$ and $y$, we conclude that $\pv H_u\subseteq\pv G_2$. The reverse
inclusion is a particular case of a general result,
namely~\cite[Proposition~5.6]{Almeida:2001b}. It is well know that
each $\mu^n(x)$ is a cube-free word, in the sense that no nonempty
factor is a cube \cite{Lothaire:1983}. Hence, the same is true of the
pseudoword~$u$. On the other hand, $u$ is a regular element of the
semigroup \Om AS, which entails that it satisfies
condition~\eqref{item:cr-2} of Theorem~\ref{t:cr}. Thus, $\pv{CR}(\pv
G_2)$ is neither $\{u,\omega\}$ nor $\{2^\omega,\omega\}$-reducible.

\section*{Acknowledgments}

The first author acknowledges partial funding by CMUP
(UID/MAT/ 00144/2013) which is funded by FCT (Portugal) with national
(MCTES) and European structural funds (FEDER) under the partnership
agreement PT2020. %
The second author was supported by Grant 15-02862S of the Grant Agency
of the Czech Republic.

\bibliographystyle{amsplain}
\bibliography{sgpabb,ref-sgps}

\end{document}